\documentclass[]{interact}

\usepackage{epstopdf}
\usepackage[caption=false]{subfig}

\usepackage[numbers,sort&compress]{natbib}
\bibpunct[, ]{[}{]}{,}{n}{,}{,}
\makeatletter
\def\NAT@def@citea{\def\@citea{\NAT@separator}}
\makeatother

\theoremstyle{plain}
\newtheorem{theorem}{Theorem}[section]

\newtheorem{corollary}[theorem]{Corollary}

\theoremstyle{definition}

\theoremstyle{remark}
\newtheorem{remark}{Remark}

\begin{document}


\title{An Extension of The First Eigen-type Ambarzumyan theorem}

\author{
\name{Alp Arslan K\i ra\c{c}\textsuperscript{a}\thanks{CONTACT Alp Arslan K\i ra\c{c}. Email: aakirac@pau.edu.tr}}
\affil{\textsuperscript{a}Department of Mathematics, Faculty of Arts and Sciences,
Pamukkale University, 20070, Denizli, Turkey }}

\maketitle

\begin{abstract}
An extension of the first eigenvalue-type Ambarzumyan's theorem are provided for the arbitrary  self-adjoint Sturm-Liouville differential operators. The result makes a contribution to the P{\"{o}}schel-Trubowitz inverse spectral theory as well. 
\end{abstract}


\section{Introduction}

In $1929$, Ambarzumyan \cite{Ambarz} proved that if  $\{(n\pi)^{2}: n\in \mathbb{N}\cup \{0\}\}$ is the spectrum of the boundary value problem
 \begin{equation}  \label{1}
-y^{\prime\prime}(x)+q(x)y(x)=\lambda y(x),\qquad y^{\prime}(0)=y^{\prime}(1)=0
\end{equation}
with real potential $q\in L^{1}(0,1)$, then $q=0$ a.e. Clearly, if $q=0$ a.e., then the eigenvalues $\lambda_{n}=(n\pi)^{2}$, $n\in\mathbb{N}\cup \{0\}$.

We note that in Ambarzumyan's theorem the whole spectrum is specified. But then Freiling and Yurko \cite{FreilingYurko} proved that it is enough to specify only the first eigenvalue. More precisely, they proved  the following first eigenvalue-type of Ambarzumyan theorem:

\begin{theorem}\label{yurk1}
If $\lambda_{0}=\int_{0}^{1}q(x)\,dx$, then $q(x)=\lambda_{0}$ a.e.
\end{theorem}

Consider the boundary value problems $L(q)$ generated in the space $L^{2}(0, 1)$ by the following differential equation
\begin{equation}\label{q-per}
-y^{\prime\prime}(x)+q(x)y(x)=\lambda y(x)
\end{equation}
with arbitrary self-adjoint boundary conditions, where $q\in L^{1}(0,1)$ is a real-valued function.
The operator $L(q)$ is  self-adjoint and its spectrum is discrete,
real and bounded from below. We suppose that the eigenvalues of the operator $L(q)$ consist of the sequence $\{\lambda_{n}\}_{n\geq 1}$ $(\lambda_{n}\leq\lambda_{n+1},\lim_{n\rightarrow\infty}\lambda_{n}=+\infty)$ (counting with multiplicities) and $\{y_{n}\}_{n\geq 1}$ denotes the corresponding normalized eigenfunctions of the operator $L(q)$.

Let us consider another known and fixed operator $\tilde{L}:=L(\tilde{q})$ of the same domain with $L:=L(q)$ but with different potential $\tilde{q}\in L^{1}(0,1)$. Denote by $\{\tilde{\lambda}_{n}\}_{n\geq 1}$ and $\{\tilde{y}_{n}\}_{n\geq 1}$ eigenvalues
and normalized eigenfunctions of the operator $\tilde{L}$, respectively.

In \cite{Yurko2013}, Yurko provided the following generalization of Theorem \ref{yurk1} on wide classes of self-adjoint differential operators. Here $(.,.)$ denotes inner product in $L^{2}(0, 1)$ and $\hat{q}:=q-\tilde{q}$.  

\begin{theorem} \label{yurkofirst}
  Let \begin{equation}\nonumber
  \lambda_{1}-\tilde{\lambda}_{1}=(\hat{q} \tilde{y}_{1},\tilde{y}_{1}),
  \end{equation}
  where    $\tilde{y}_{1}$ is a normalized eigenfunction of $\tilde{L}$ related to the first eigenvalue $\tilde{\lambda}_{1}$. Then $q(x)=\tilde{q}(x)+ \lambda_{1}-\tilde{\lambda}_{1}$ a.e. on $(0,1)$.
\end{theorem}

Note that the proof of this theorem is based on the well-known variational principle of the smallest eigenvalue.

Recently Ashrafyan \cite{Ashrafyan2018} proved the following another generalization of the first eigenvalue-type of Ambarzumyan theorem for Sturm-Liouville problems with arbitrary self-adjoint boundary conditions.

\begin{theorem} \label{Ashrafyan}
  Let \begin{equation}\nonumber
  \lambda_{1}-\tilde{\lambda}_{1}=ess \inf\hat{q} \quad \text{or}\quad \lambda_{1}-\tilde{\lambda}_{1}=ess \sup\hat{q},
  \end{equation}
  where    $\tilde{\lambda}_{1}$ is the first eigenvalue $\tilde{L}$. Then $q(x)=\tilde{q}(x)+ \lambda_{1}-\tilde{\lambda}_{1}$ a.e. on $(0,1)$.
\end{theorem}
As in the proof of Theorem \ref{yurkofirst}, the above uniqueness theorem is also provided by using the property of the smallest eigenvalue. That is, the proof of Theorem \ref{Ashrafyan} is based on the Sturm oscillation theorem that the first eigenfunction has no zeros on interval $(0,1)$. 
\section{Main results}

The main result of this paper is as follows. Note that, to obtain the following extension theorem, it is enough to have information about the arbitrary eigenvalue $\lambda_{n}$ instead of the first eigenvalue only.
\begin{theorem} \label{main}
  Let, for some $n$,
  \begin{equation}\nonumber
  \lambda_{n}-\tilde{\lambda}_{n}=(\hat{q} \tilde{y}_{n},\tilde{y}_{n})
  \end{equation}
  and
  \begin{equation}\nonumber
  \lambda_{n}-\tilde{\lambda}_{n}=ess \inf\hat{q} \quad \text{or}\quad \lambda_{n}-\tilde{\lambda}_{n}=ess \sup\hat{q},
  \end{equation}
  where    $\tilde{y}_{n}$ is a normalized eigenfunction of $\tilde{L}$ related to the eigenvalue $\tilde{\lambda}_{n}$. Then $q(x)=\tilde{q}(x)+ \lambda_{n}-\tilde{\lambda}_{n}$ a.e. on $(0,1)$.
\end{theorem}
\begin{proof}
It follows from the first assumption that
\begin{equation}
((\hat{q}+\tilde{\lambda}_{n}- \lambda_{n}) \tilde{y}_{n},\tilde{y}_{n})=0.
\end{equation}
Since $\tilde{y}_{n}$ is a normalized eigenfunctions of the operator $\tilde{L}$ corresponding to the eigenvalue $\tilde{\lambda}_{n}$, we obtain that $\tilde{y}_{n}$ has at most finitely many isolated zero points in $(0,1)$. Therefore, the measure of the set of zero points is $0$. Hence, using this and $(\hat{q}+\tilde{\lambda}_{n}- \lambda_{n}) \tilde{y}_{n}^{2}\in  L^{1}(0,1)$, we get $\hat{q}=\lambda_{n}-\tilde{\lambda}_{n}$ a.e. on $(0,1)$.
\end{proof}

From Theorem \ref{main}, one can easily verify the following assertion. 
\begin{theorem} \label{main1}
  Let the assumptions of Theorem \ref{main} be valid and let $\int_0^1 q(x)dx=\int_0^1 \tilde{q}(x)dx$. Then $q(x)=\tilde{q}(x)$ a.e. on $(0,1)$.
\end{theorem}

\subsection{Example}
Let us give an example to illustrate Theorem \ref{main}.

Consider the Dirichlet boundary value problem
\begin{equation}  \label{dirich}
-y^{\prime\prime}(x)+q(x)y(x)=\lambda y(x),\qquad y(0)=y(1)=0.
\end{equation}
Let $\tilde{q}(x)\equiv 0$. Then $\tilde{\lambda}_{n}=(n\pi)^2$, $\tilde{y}_{n}=\sqrt{2}\sin n\pi x$ for some $n\geq 1$ and Theorem \ref{main} implies the following assertion.
\begin{corollary} \label{co1}
  Let, for some $n$,
  \begin{equation}\label{c1}
  \lambda_{n}-(n\pi)^2=2\int_{0}^1 q(x)\sin^{2} n\pi x\, dx
  \end{equation}
  and
  \begin{equation}\nonumber
  \lambda_{n}-(n\pi)^2=ess \inf q \quad \text{or}\quad \lambda_{n}-(n\pi)^2=ess \sup q.
  \end{equation}
   Then $q(x)=\lambda_{n}-(n\pi)^2$ a.e. on $(0,1)$.
\end{corollary}

And, Theorem \ref{main1} implies the following assertion. The assertion makes a contribution to the P{\"{o}}schel-Trubowitz inverse spectral theory.
\begin{corollary} \label{co2}
  Let the assumptions of Corollary \ref{co1} be valid and let $\int_0^1 q(x)dx=0$. Then $q(x)=0$ a.e. on $(0,1)$.
\end{corollary}
\begin{remark}
In \cite{Poschel}, P{\"{o}}schel and Trubowitz showed that, for the Dirichlet problem, there is an infinite dimensional set of $L^{2}(0, 1)$ potentials with the same Dirichlet zero spectrum $\sigma:=\{(n\pi)^{2}: n\in \mathbb{N}\}$ as $q=0$ a.e. That is, if the spectrum is $\sigma$, then the potential is not necessarily zero. Thus Ambarzumyan's theorem is not valid.
\end{remark}
Arguing as in the paper \cite{chern2001} (see p. 337), from P{\"{o}}schel and Trubowitz \cite{Poschel}, an isospectral $L^{2}(0, 1)$ potential can be written in the form $q=u+e(u)$, where $u$ is odd part and $e$ is even part which can be uniquely expressed as a function of its odd part ($e=e(u)$).

Let $q$ be an isospectral potential with the Dirichlet zero spectrum $\sigma=\{(n\pi)^{2}: n\in \mathbb{N}\}$. Then  $\int_0^1 q(x)dx=0$ and from (\ref{c1}), we get for some $n$
\begin{equation}\label{zero}
0=2\int_{0}^1 q(x)\sin^{2} n\pi x\, dx=\int_0^1 q(x)dx-\int_0^1 q(x)\cos 2n\pi x\,dx.
  \end{equation}
Hence, for some $n$,
\begin{equation}\nonumber
\int_0^1 q(x)\cos 2n\pi x\,dx=0
  \end{equation}
which implies that the n-th even Fourier coefficient $a_n$ has the following equalities
\begin{equation}\nonumber
a_n=\int_0^1 q(x)\cos 2n\pi x\,dx=\int_0^1 e(u)\cos 2n\pi x\,dx=0.
  \end{equation}
Then $a_n=0$ for all $n$. Similarly, all the odd Fourier coefficients vanish. Therefore, the even part of an isospectral potential does not vanish. Consequently, from Corollary \ref{co2}, for the isospectral potential $q$ (i.e. $q\neq 0$), we get, for all $n$,
\begin{equation}\nonumber
2\int_{0}^1 q(x)\sin^{2} n\pi x\, dx\neq 0,
  \end{equation}
that is, we get from (\ref{zero}), for all $n$, 
\begin{equation}\nonumber
a_n=\int_{0}^1 q(x)\cos 2n\pi x\, dx\neq 0.
  \end{equation}
Thus, the present paper supplements the P{\"{o}}schel-Trubowitz inverse spectral theory.


\end{document}